\documentclass[11pt,reqno]{amsart}
\usepackage{amsmath,amssymb,amscd,amsxtra,amsfonts,mathrsfs}
\usepackage{bm,epsf,graphicx,epsfig,color,latexsym,cite,cases}

\newtheorem{theorem}{Theorem}[section]
\newtheorem{corollary}[theorem]{Corollary}

\newtheorem{proposition}[theorem]{Proposition}
\newtheorem{remark}[theorem]{Remark}
\newtheorem{assumption}[theorem]{Assumption}

\numberwithin{equation}{section}

\begin{document}

\title[an inverse random source problem for Maxwell's equations]{An inverse
random source problem for Maxwell's equations}

\author{Peijun Li}
\address{Department of Mathematics, Purdue University, West Lafayette, Indiana 47907, USA}
\email{lipeijun@math.purdue.edu}

\author{Xu Wang}
\address{Department of Mathematics, Purdue University, West Lafayette, Indiana 47907, USA}
\email{wang4191@purdue.edu}

\thanks{The research is supported in part by the NSF grant DMS-1912704.}



\begin{abstract}
This paper is concerned with an inverse random source problem for the
three-dimensional time-harmonic Maxwell equations. The source is assumed to be a
centered complex-valued Gaussian vector field with correlated components, and
its covariance operator is a pseudo-differential operator. The well-posedness of
the direct source scattering problem is established and the regularity of the
electromagnetic field is given. For the inverse source scattering problem, the
micro-correlation strength matrix of the covariance operator is shown to be
uniquely determined by the high frequency limit of the expectation of the
electric field measured in an open bounded domain disjoint with the support of
the source. In particular, we show that the diagonal entries of the strength
matrix can be uniquely determined by only using the amplitude of the electric
field. Moreover, this result is extended to the almost surely sense by deducing
an ergodic relation for the electric field over the frequencies. 
\end{abstract}

\maketitle

\section{Introduction}

Inverse source scattering problems are to infer the information of the radiating
sources by using the measured wave fields generated by the unknown sources.
These problems arise naturally and have significant applications in many
scientific areas such as biomedical engineering, medical imaging, and optical
tomography \cite{ABF02, I90, FKM04, NOHTA07}. They have attracted much attention
by many researchers in both of the engineering and mathematical communities.
Consequently, a great number of computational and mathematical results are
available \cite{BLT10, BLZ20, IL18, AM06}. In particular,
modeled by Maxwell's equations, the inverse source scattering problem for
electromagnetic waves is an important research subject not only from the
viewpoint of engineering and industrial applications but also from
the mathematical aspect. For instance, the model can be used to determine the
source currents in the brain based on the electric or magnetic measurements on
the surface of the human head \cite{AM06}. As for the mathematical studies, we
refer to \cite{AM06} for the unique recovery of surface current density, to
\cite{V12, WMGL18, WSGLL19} for the unique recovery of volume current density,
and to \cite{BLZ20} for the stability analysis on the inverse source problems
for elastic and electromagnetic waves. 

So far, all the sources have been considered to be deterministic functions in
the existing mathematical models for the inverse electromagnetic source
scattering problem. However, in many practical situations, the source of
the system should be described by a random field instead of a deterministic
function due to the unpredictability of the surrounding environment or
uncertainties associated with the source itself \cite{D91}.

Compared with the deterministic counterparts, if the source is a random
field whose covariance operator is not regular enough, then the source would be
too rough to exist point-wisely. In this case, the source should be understood 
as a distribution, and the corresponding problem and its solution should be
studied in the distribution sense. For instance, for a $d$-dimensional problem,
if the random source is micro-locally isotropic with order $s$, i.e., its
covariance operator is a pseudo-differential operator with the principal symbol
$\phi(x)|\xi|^{-2s}$, where  $s\in(0, \frac d2]$ is a real number and
$\phi\in C_0^{\infty}(\mathbb{R}^d)$ is a positive function representing the
micro-correlation strength of the random source, then the source is a
distribution in the Sobolev space $W^{s-\frac d2-\epsilon}(\mathbb{R}^d)$ for
any $\epsilon>0$ (cf. \cite{LW}). There are already some mathematical studies on
the inverse random source problems for acoustic and elastic waves, which are to
recover the strength $\phi$ by using measured wave fields in a domain which has
a positive distance to the support of the source. If $s=0$, then the source is
as rough as a random field of the form $\sqrt{\phi}\dot{W}$ with $\dot{W}$ being
the white noise (cf. \cite{LW}). When the source is modeled by a white noise
type random field, the It\^o isometry can be resorted to recover the variance of
the random source. We refer to \cite{BCL16, LCL18} and \cite{BCL17} on the study
of the inverse acoustic and elastic source scattering problems, respectively. If
$s\in(0,\frac d2+1)$, the It\^o isometry is not valid any more since the
increments of the random source may be correlated. It turns out that the
micro-local analysis is effective to handle such a random source.
For the inverse acoustic random source scattering problems, we refer to
\cite{LHL} for the case $s\in[\frac d2,\frac d2+1)$ and to \cite{LW} for the
case $s\in(0,\frac d2+1)$. The results can be found in \cite{LHL, LL} on the
inverse elastic random source scattering problems with $s\in[\frac d2,\frac
d2+1)$. We refer to \cite{LLM1, LLM2} for related inverse problems on the
stochastic Schr\"odinger equation. To the best of our knowledge, the inverse
random source problem for Maxwell's equations is
completely open! This work initializes the mathematical study on the direct and
inverse source scattering problems for the stochastic Maxwell equations driven
by a random electric current density. 

In this paper, we consider the three-dimensional time-harmonic stochastic
Maxwell equations
\begin{equation}\label{eq:Maxwell}
\nabla\times\boldsymbol{E}={\rm i}k\boldsymbol{H},\quad 
\nabla\times\boldsymbol{H}={\rm i}k\boldsymbol{E}+\boldsymbol{J},
\end{equation}
where $k>0$ is the wavenumber, $\boldsymbol E$ and $\boldsymbol H$ are the
electric field and the magnetic field, respectively, and $\boldsymbol J$ is the
electric current density, which is assumed to be a complex-valued random vector
field defined on a complete probability space $(\Omega,\mathcal{F},\mathbb{P})$
with a compact support $\mathcal{O}\subset\mathbb{R}^3$. Moreover, the source
$\boldsymbol{J}$ is assumed to be micro-locally isotropic such that its
covariance operator is a pseudo-differential operator with the
principal symbol given by $A(x)|\xi|^{-2s}$, $s\in(0,\frac 52)$, where the
complex-valued matrix $A\in C_0^{\infty}(\mathbb R^3;\mathbb C^{3\times 3})$
describes the micro-correlation strength of the random source and the entries
are assumed to be smooth functions with compact supports contained in $\mathcal
O$. Hence we consider a more general principal symbol than that studied in
\cite{LHL, LL, LW}, where the principal symbol is characterized by a scalar
real-valued smooth function. Given the electric current density $\boldsymbol J$,
the direct scattering problem is to study the well-posedness of
\eqref{eq:Maxwell}; the inverse scattering problem is to determine $\boldsymbol
J$ from a knowledge of the electric field $\boldsymbol E$. Due to such a random
$\boldsymbol J$, both of the direct and inverse scattering problems are
challenging. 
  
The work contains three contributions. First, by considering an
equivalent problem, the well-posedness is established for \eqref{eq:Maxwell} in
the distribution sense. The regularity is given for both the electric field
$\boldsymbol E$ and the magnetic field $\boldsymbol H$. A key ingredient is to
find an appropriate function space for the electric current density
$\boldsymbol J$, which is required to satisfy a divergence free condition in
the weak sense. Second, we show that the micro-correlation strength matrix $A$
is uniquely determined by the high frequency limit of the expectation
of the electric field measured in a bounded open domain disjoint with the
support of $\boldsymbol J$. The result also implies that the diagonal entries of
the strength matrix $A$ can be uniquely determined by the high frequency limit
of the amplitude of the electric field, which is known as the phaseless data.
Third, if only the amplitude of the electric field is available, then we
show that the diagonal entries of the strength matrix $A$ can be uniquely
recovered by the energy of the electric field averaged over the frequency band
at a single realization of the random source, which indicates that it is
statistically stable to recover the strength matrix. The idea is to deduce an
ergodic relation for the electric field over the frequencies in order to obtain
such a strong result.

The paper is organized as follows. In Section 2, we address the direct source
scattering problem. The properties and assumptions are introduced for the
random source; the well-posedness of \eqref{eq:Maxwell} and the regularity of
the electromagnetic field are examined. Sections 3 and 4 are devoted to the
inverse source scattering problem. In Section 3, we discuss the uniqueness to
recover the micro-correlation strength matrix by using the expectation of the
electric field; while in Section 4, we present the uniqueness result by using
the amplitude of the electric field at a single path. The paper is concluded
with some general remarks and directions for future work in Section 5.

\section{Direct scattering problem}

In this section, we introduce some basic notation for complex isotropic
Gaussian random fields and establish the well-posedness for the direct
scattering problem if the current density is a complex-valued isotropic
Gaussian random field. 

\subsection{Complex isotropic Gaussian random fields}
\label{sec:comRF}

Let $\boldsymbol{J}(x)$ be a complex-valued Gaussian random vector field.
It can be determined by the mean
$\boldsymbol{m}(x)=\mathbb{E}[\boldsymbol{J}(x)]$,
the covariance 
\[
C_{\boldsymbol{J}}( x, y)=\mathbb{E}\left[(\boldsymbol{J}(
x)-\mathbb{E}[\boldsymbol{J}( x)])(\overline{\boldsymbol{J}(
y)-\mathbb{E}[\boldsymbol{J}( y)]})^\top\right],
\] 
and the relation 
\[
R_{\boldsymbol{J}}( x, y)=\mathbb{E}\left[(\boldsymbol{J}(
x)-\mathbb{E}[\boldsymbol{J}( x)])(\boldsymbol{J}( y)-\mathbb{E}[\boldsymbol{J}(
y)])^\top\right] 
\] 
if they exist. It is easy to verify the following properties for the complex-valued covariance
and relation matrix functions: for any $x, y\in\mathbb R^3$,
\begin{itemize}

\item[(i)] $C^*_{\boldsymbol{J}}(y, x)=\overline{C_{\boldsymbol{J}}( y,
x)}^\top=C_{\boldsymbol{J}}( x,
y)$;

\item[(ii)] $\overline{R_{\boldsymbol{J}}( x, y)}=R_{\overline{\boldsymbol{J}}}(
x, y)$ and $R_{\boldsymbol{J}}( y, x)^\top=R_{\boldsymbol{J}}( x, y)$;

\item[(iii)] $C_{\boldsymbol{J}}( x, y)=R_{\boldsymbol{J}}( x, y)$ if
$\boldsymbol{J}(x)$ is real-valued;

\item[(iv)] $R_{\boldsymbol{J}}( x, y)={\bf 0}$ if the real and imaginary parts
of $\boldsymbol{J}$ are independent and identically distributed.

\end{itemize}

For a complex-valued Gaussian random vector $\boldsymbol{Z}=\boldsymbol{X}+{\rm
i}\boldsymbol{Y}$, the variance matrices of $\boldsymbol{X}$ and
$\boldsymbol{Y}$, and the covariance matrices between $\boldsymbol{X}$ and
$\boldsymbol{Y}$ are uniquely determined by the covariance and relation of
$\boldsymbol{Z}$, and vice versa. More precisely, let 
$V_{\boldsymbol{X}\boldsymbol{X}}$ and $V_{\boldsymbol{Y}\boldsymbol{Y}}$ be the
variance matrices of $\boldsymbol{X}$ and $\boldsymbol{Y}$, and let
$V_{\boldsymbol{X}\boldsymbol{Y}}$ and $V_{\boldsymbol{Y}\boldsymbol{X}}$ be the
covariance matrices between $\boldsymbol{X}$ and $\boldsymbol{Y}$. 
Denote by $C$ and $R$ the covariance and relation matrices of $\boldsymbol{Z}$.
Then it is easy to note that 
\begin{align*}
V_{\boldsymbol{X}\boldsymbol{X}}=\frac12\Re[C+R],\quad
V_{\boldsymbol{Y}\boldsymbol{Y}}=\frac12\Re[C-R],\quad
V_{\boldsymbol{X}\boldsymbol{Y}}=\frac12\Im[R-C],\quad
V_{\boldsymbol{Y}\boldsymbol{X}}=\frac12\Im[R+C]. 
\end{align*}
where $\Re[\cdot]$ and $\Im[\cdot]$ stand for the real and imaginary parts of a
complex number or matrix, respectively. Conversely, we have from a simple
calculation that 
\begin{align*}
C=V_{\boldsymbol{X}\boldsymbol{X}}+V_{\boldsymbol{Y}\boldsymbol{Y}}+{\rm
i}(V_{\boldsymbol{Y}\boldsymbol{X}}-V_{\boldsymbol{X}\boldsymbol{Y}}),\quad
R=V_{\boldsymbol{X}\boldsymbol{X}}-V_{\boldsymbol{Y}\boldsymbol{Y}}+{\rm
i}(V_{\boldsymbol{Y}\boldsymbol{X}}+V_{\boldsymbol{X}\boldsymbol{Y}}).
\end{align*}

If $\boldsymbol{J}$ is not regular enough, the covariance and relation
matrix functions may not exist point-wisely. Hence, it is necessary to give
rigorous definitions of the covariance and the relation of $\boldsymbol{J}$. Let
$\mathcal{D}:=\mathcal{D}(\mathbb{R}^3)$ be the space of test functions on
$\mathbb{R}^3$, which is $C_0^{\infty}(\mathbb{R}^3)$ equipped with a locally
convex topology. Denote by $\mathcal{D}':=\mathcal{D}'(\mathbb{R}^3)$ the space
of distributions on $\mathbb{R}^3$, which is the dual space of $\mathcal{D}$
equipped with the weak-star topology. Denote by $\langle\cdot,\cdot\rangle$ the
dual product between $(\mathcal{D}')^3$ and ${\mathcal D}^3$.
Then the derivative of a distribution $\boldsymbol{\psi}\in({\mathcal
D}')^3$ is defined by
\[
\langle \partial_{x_j}\boldsymbol{\psi},\boldsymbol{\varphi}\rangle=-\langle
\boldsymbol{\psi},\partial_{x_j}\boldsymbol{\varphi}\rangle\quad\forall\,
\boldsymbol{\varphi}\in{\mathcal D}^3
\]
for $j=1,2,3.$ We refer to \cite{AF03} and references cited therein
for more details about distributions. Define the covariance operator
$Q_{\boldsymbol{J}}^C$ and the relation operator
$Q_{\boldsymbol{J}}^R$ by
\begin{align*}
\langle
Q_{\boldsymbol{J}}^C\boldsymbol{\varphi},\boldsymbol{\psi}\rangle:=&\int_{
\mathbb{R}^3}\int_{\mathbb{R}^3}\boldsymbol \psi^*( y)
\mathbb{E}\left[\boldsymbol{J}( x)\boldsymbol{J}^*(
y)\right]\boldsymbol{\varphi}( x)d x d y\nonumber\\
=&\int_{\mathbb{R}^3}\int_{\mathbb{R}^3}\boldsymbol \psi^*( y)
C_{\boldsymbol{J}}(x, y)\boldsymbol{\varphi}(x)d x d y
\end{align*}
and 
\begin{align*}
\langle
Q_{\boldsymbol{J}}^R\boldsymbol{\varphi},\boldsymbol{\psi}\rangle:=&\int_{
\mathbb{R}^3}\int_{\mathbb{R}^3}\boldsymbol \psi^*( y)
\mathbb{E}[\boldsymbol{J}( x)\boldsymbol{J}(
y)^\top]\boldsymbol{\varphi}( x)d x d y\nonumber\\
=&\int_{\mathbb{R}^3}\int_{\mathbb{R}^3}\boldsymbol\psi^*( y)
R_{\boldsymbol{J}}( x, y)\boldsymbol{\varphi}( x)d x d y
\end{align*}
for any $\boldsymbol{\varphi},\boldsymbol{\psi}\in{\mathcal D}^3$, where the
star denotes the complex conjugate. 

Hereafter, we use the notation $W^{r,p}:=W^{r,p}(\mathbb{R}^3)$ and
$C^{r,\alpha}:=C^{r,\alpha}(\mathbb{R}^3)$ for simplicity. For any space $X$, we
denote by $\boldsymbol{X}$ the Cartesian product vector space $X^3$ for
convenience. Without loss of
generality, we may assume that the current density $\boldsymbol J$ is a centered
Gaussian random field. If not, it is essentially a deterministic inverse source
problem to determine the nonzero mean, which has been well studied in
\cite{BLZ20}. In addition, the current density $\boldsymbol J$ is required to
satisfying the following conditions. 

\begin{assumption}\label{as:J}
Let $\boldsymbol{J}\in\boldsymbol{\mathcal{D}}'$ be a complex-valued isotropic centered
Gaussian random vector field compactly supported in 
$\mathcal{O}\subset\mathbb{R}^3$ with the covariance kernel 
$C_{\boldsymbol{J}}(x, y)$ and the relation kernel $R_{\boldsymbol{J}}( x, y)$
depending only on $|x- y|$. Assume that 
\begin{itemize}
\item[(i)] the real and imaginary parts of $\boldsymbol{J}$ are independent and
identically distributed with the relation operator $Q_{\boldsymbol{J}}^R=0$;
\item[(ii)] the covariance operator $Q_{\boldsymbol{J}}^C$ defined through the
kernel $C_{\boldsymbol{J}}$ is a pseudo-differential operator of order
$s\in[0,\frac52)$, which implies that $Q_{\boldsymbol{J}}^C$ has a principal
symbol $A(x)|\xi|^{-2s}$, where $A(x)\in C_0^{\infty}(\mathbb R^3;\mathbb C^{3\times 3})$ is a
smooth matrix function with a compact support contained in $\mathcal{O}$.
\end{itemize}
\end{assumption}

Given the current density $\boldsymbol J$ satisfying Assumption \ref{as:J}, the
direct scattering problem is to study the well-posedness of  Maxwell's
equations \eqref{eq:Maxwell}. We intend to answer the following questions: what
are the conditions of $\boldsymbol J$ such that Maxwell's equations
\eqref{eq:Maxwell} admit a unique solution $(\boldsymbol E, \boldsymbol H)$?
What are the regularity for $\boldsymbol E$ and $\boldsymbol H$ if there is a
unique solution? For the inverse scattering problem, the goal is not to
determine the random current density $\boldsymbol J$ but to determine the matrix $A$, which
represents the micro-correlation strength of the current density $\boldsymbol
J$, from a knowledge of the measured electric field $\boldsymbol E$. We are
concerned with the uniqueness for the inverse scattering problem: can $A$ or
what part of $A$ be uniquely determined by the available data? To give a
detailed explanation of $A$, we rewrite $\boldsymbol{J}=(J_1,J_2,J_3)^\top$ by
its components. Then a simple calculation yields that 
\[
\boldsymbol{J}( x)\boldsymbol{J}^*( y)\stackrel{d}{=}\left[\begin{array}{ccc}
J_1( x)\overline{J_1( y)}&\cdots&J_1( x)\overline{J_3( y)}\\
\vdots&\ddots &\vdots\\
J_3( x)\overline{J_1( y)}&\cdots&J_3( x)\overline{J_3( y)}
\end{array}\right],
\]
where $\stackrel{d}{=}$ means ``equals in distribution". 
As a result, each entry in $A(x)$ is determined by the strength of
covariance operator between $J_j$ and $J_l$ with $j,l=1,2,3$.

\subsection{Well-posedness} 

If the current density $\boldsymbol{J}\in\boldsymbol{\mathcal{D}}'$ is a
distribution, then Maxwell's equations \eqref{eq:Maxwell} do not hold
point-wisely any more. To establish the well-posedness of \eqref{eq:Maxwell} in
some proper sense, we impose the weak Silver--M\"uller radiation condition
\[
\lim_{r\to\infty}\int_{| x|=r}\left(\boldsymbol{H}\times\frac{ x}{|
x|}-\boldsymbol{E}\right)\cdot\boldsymbol{\phi} ds=0\quad
\forall\,\boldsymbol{\phi}\in\boldsymbol{\mathcal{D}}, 
\]
which characterizes the behavior of solutions to \eqref{eq:Maxwell} at
infinity.

Eliminating the magnetic field $\boldsymbol{H}$ from \eqref{eq:Maxwell},
multiplying a test function $\boldsymbol{\phi}\in \boldsymbol{\mathcal{D}}$,
and integrating over $\mathbb R^3$, we get
\begin{align*}
\int_{\mathbb{R}^3}(\nabla\times(\nabla\times\boldsymbol{E}))\cdot\boldsymbol{
\phi} d x-k^2\int_{\mathbb{R}^3}\boldsymbol{E}\cdot\boldsymbol{\phi} d x={\rm
i}k\int_{\mathbb{R}^3}\boldsymbol{J}\cdot\boldsymbol{\phi} d
x\quad\forall\,\boldsymbol{\phi}\in \boldsymbol{\mathcal{D}},
\end{align*}
which, by derivatives of distributions, leads to 
\begin{align}\label{eq:disE}
\int_{\mathbb{R}^3}(-\Delta-k^2)\boldsymbol{E}\cdot\boldsymbol{\phi} d
x-\int_{\mathbb{R}^3}(\nabla\cdot\boldsymbol{E})(\nabla\cdot\boldsymbol{\phi}) d
x={\rm i}k\int_{\mathbb{R}^3}\boldsymbol{J}\cdot\boldsymbol{\phi} d
x\quad\forall\,\boldsymbol{\phi}\in \boldsymbol{\mathcal{D}}.
\end{align}

Moreover, for any $\boldsymbol{\phi}\in\boldsymbol{\mathcal{D}}$, it
follows from the second equation in \eqref{eq:Maxwell} that we get
$\nabla(\nabla\cdot\boldsymbol{\phi})\in\boldsymbol{\mathcal{D}}$ and hence
\[
\int_{\mathbb{R}^3}(\nabla\times\boldsymbol{H})\cdot(\nabla(\nabla\cdot\boldsymbol{\phi}))d x=-{\rm i}k\int_{\mathbb{R}^3}\boldsymbol{E}\cdot(\nabla(\nabla\cdot\boldsymbol{\phi}))d x+\int_{\mathbb{R}^3}\boldsymbol{J}\cdot(\nabla(\nabla\cdot\boldsymbol{\phi}))d x,
\]
which implies
\begin{align}\label{eq:condE}
\int_{\mathbb{R}^3}(\nabla\cdot\boldsymbol{E})(\nabla\cdot\boldsymbol{\phi}) d
x=\frac{\rm i}k\int_{\mathbb{R}^3}\boldsymbol{J}
\cdot(\nabla(\nabla\cdot\boldsymbol{\phi}))d
x\quad\forall\,\boldsymbol{\phi}\in\boldsymbol{\mathcal{D}}.
\end{align}

Define the space
\[
\mathbb{X}:=\Big\{\boldsymbol{U}\in\boldsymbol{\mathcal D}':\int_{\mathbb{R}
^3}\boldsymbol{U}\cdot\left(\nabla(\nabla\cdot\boldsymbol{\phi})\right) d
x=0 \quad \forall\,\boldsymbol{\phi}\in \boldsymbol{\mathcal{D}}\Big\}.
\]
Apparently, $\mathbb{X}$ is non-empty since all divergence free vector fields
are included. If $\boldsymbol{J}\in\mathbb{X}$, we obtain from
\eqref{eq:disE}--\eqref{eq:condE} that 
\begin{align*}
\int_{\mathbb{R}^3}\left[(\Delta+k^2)\boldsymbol{E}+{\rm
i}k\boldsymbol{J}\right]\cdot\boldsymbol{\phi} d
x=0\quad\forall\,\boldsymbol{\phi}\in\boldsymbol{\mathcal{D}},
\end{align*}
which indicates that the following Helmholtz equation holds in the distribution
sense: 
\begin{align}\label{eq:Helm}
(\Delta+k^2)\boldsymbol{E}=-{\rm i}k\boldsymbol{J}.
\end{align}

\begin{theorem}\label{tm:well-posed}
Let $p\in(\frac32,2]$, $s\in(\frac 3p-\frac12,\frac32]$ and $H=s-\frac32\in(\frac3p-2,0]$. Assume that
$\boldsymbol{J}\in\mathbb{X}\cap\boldsymbol{W}^{H-\epsilon,p}_{comp}$ for any
$\epsilon>0$ with a compact support contained in $\mathcal{O}$. Then
\eqref{eq:Helm} admits a unique solution
\begin{align*}
\boldsymbol{E}( x)={\rm i}k\int_{\mathbb{R}^3}\Phi_k( x, y)\boldsymbol{J}( y)d y\quad\text{a.s.}
\end{align*}
in $\mathbb{X}\cap\boldsymbol{W}_{loc}^{-H+\epsilon,q}$ with $q$ satisfying $\frac1p+\frac1q=1$ and 
\[
\Phi_k( x, y)=\frac{e^{{\rm i}k| x- y|}}{4\pi| x- y|}
\]
being the fundamental solution for the three-dimensional Helmholtz equation. 
\end{theorem}

\begin{proof}
It has been shown in \cite{LW} that the scalar Helmholtz equation in
$\mathbb{R}^3$ has a unique solution in $W^{-H+\epsilon,q}$, which implies the
well-posedness of \eqref{eq:Helm} in $\boldsymbol{W}^{-H+\epsilon,q}$. It then
suffices to show $\boldsymbol{E}\in\mathbb{X}$. In fact, noting
$\nabla_{ x}\Phi_k( x, y)=-\nabla_{ y}\Phi_k( x, y)$, we have for any
$\boldsymbol{\phi}\in\boldsymbol{\mathcal{D}}$ that 
\begin{align}\label{wp-s1}
\int_{\mathbb{R}^3}\boldsymbol{E}( x)\cdot\nabla_{ x}\left(\nabla_{
x}\cdot\boldsymbol{\phi}\right)d x
&= {\rm i}k\int_{\mathbb{R}^3}\left[\int_{\mathbb{R}^3}\Phi_k( x,
y)\nabla_{ x}\left(\nabla_{ x}\cdot\boldsymbol{\phi}\right)d
x\right]\cdot\boldsymbol{J}( y)d y\notag\\
&= {\rm i}k\int_{\mathbb{R}^3}\nabla_{ y}\left[\int_{\mathbb{R}^3}\Phi_k( x,
y)\left(\nabla_{ x}\cdot\boldsymbol{\phi}\right)d x\right]\cdot\boldsymbol{J}(
y)d y\notag\\
&= -{\rm i}k\int_{\mathbb{R}^3}\nabla_{
y}\left[\int_{\mathbb{R}^3}\left(\nabla_{ x}\Phi_k( x,
y)\right)\cdot\boldsymbol{\phi}( x)d x\right]\cdot\boldsymbol{J}( y)d y\notag\\
&= {\rm i}k\int_{\mathbb{R}^3}\nabla_{ y}\left[\int_{\mathbb{R}^3}\left(\nabla_{
y}\Phi_k( x, y)\right)\cdot\boldsymbol{\phi}( x)d x\right]\cdot\boldsymbol{J}(
y)d y\notag\\
&= {\rm i}k\int_{\mathcal{O}}\left(\nabla_{ y}\left(\nabla_{
y}\cdot\left[\int_{\mathbb{R}^3}\Phi_k( x, y)\boldsymbol{\phi}( x)d
x\right]\right)\right)\cdot\boldsymbol{J}( y)d y.
\end{align}
Let
\[
 \boldsymbol{f}( y)=\int_{\mathbb{R}^3}\Phi_k( x, y)\boldsymbol{\phi}( x)d
x,\quad  y\in\mathcal{O}
\]
and choose a sufficiently large ball $B$ such that $\mathcal{O}\subsetneq B$.
Define a smooth extension $\widetilde{\boldsymbol{f}}$ on $\mathbb{R}^3$ such
that
\[
\widetilde{\boldsymbol{f}}( y)=
\begin{cases}
\boldsymbol{f}( y),&\quad y\in\mathcal{O},\\
\boldsymbol{0},&\quad y\in\mathbb{R}^3\setminus \overline{B}.
\end{cases}
\]
It is easy to note that
$\widetilde{\boldsymbol{f}}\in\boldsymbol{\mathcal{D}}$. Since 
$\boldsymbol{J}\in\mathbb{X}$, we have from \eqref{wp-s1} that
\begin{align*}
\int_{\mathbb{R}^3}\boldsymbol{E}( x)\cdot\nabla_{x}\left(\nabla_{
x}\cdot\boldsymbol{\phi}\right)d x
&={\rm i}k\int_{\mathcal{O}}\left(\nabla_{y}\left(\nabla_{
y}\cdot\boldsymbol{f}( y)\right)\right)\cdot\boldsymbol{J}( y)d y\\
&={\rm i}k\int_{\mathbb{R}^3}\left(\nabla_{ y}\left(\nabla_{
y}\cdot\widetilde{\boldsymbol{f}}( y)\right)\right)\cdot\boldsymbol{J}( y)d y\\
&=0,
\end{align*}
which completes the proof.
\end{proof}

\begin{corollary}
Under the assumptions in Theorem \ref{tm:well-posed}, the Helmholtz equation
\eqref{eq:Helm} together with \begin{align}\label{eq:Helm2}
\nabla\times\boldsymbol{E}={\rm i}k\boldsymbol{H}
\end{align}
is equivalent to Maxwell's equations \eqref{eq:Maxwell} in the distribution
sense. 

Moreover, it holds
$\boldsymbol{H}\in \left(\boldsymbol{W}^{H-\epsilon,p}(curl)\right)'$ which
is the dual space of $\boldsymbol{W}^{H-\epsilon,p}(curl)$ equipped with
norm
\[
\|\boldsymbol{h}\|_{\boldsymbol{W}^{H-\epsilon,p}(curl)}=\left(\|\boldsymbol{h}
\|_{\boldsymbol{W}^{H-\epsilon,p}}^2+\|\nabla\times\boldsymbol{h}\|_{\boldsymbol
{W}^{H-\epsilon,p}}^2\right)^{\frac12}.
\]
\end{corollary}

\begin{proof}
Based on the above discussions, it has been shown that any solution of
\eqref{eq:Maxwell} is also a solution of \eqref{eq:Helm}--\eqref{eq:Helm2}. Next
we show that if
$\boldsymbol{J}\in\mathbb{X}\cap\boldsymbol{W}^{H-\epsilon,p}_{comp}$ and
$\boldsymbol{E}\in\mathbb{X}\cap\boldsymbol{W}^{-H+\epsilon,q}$ is a solution of
\eqref{eq:Helm}--\eqref{eq:Helm2} as stated in Theorem \ref{tm:well-posed}, then
$\boldsymbol{E}$ also solves \eqref{eq:Maxwell}. 

Noting $\boldsymbol{E}\in\mathbb{X}$ and using \eqref{eq:Helm} and
\eqref{eq:Helm2}, we get for any $\boldsymbol{\phi}\in\boldsymbol{\mathcal{D}}$
that
\begin{align*}
-{\rm i}k\int_{\mathbb{R}^3}\boldsymbol{J}\cdot\boldsymbol{\phi}d
x&= \int_{\mathbb{R}^3}(\Delta+k^2)\boldsymbol{E}\cdot\boldsymbol{\phi}d x\\
&= \int_{\mathbb{R}^3}\left[-\nabla\times(\nabla\times\boldsymbol{E}
)+\nabla(\nabla\cdot\boldsymbol{E})+k^2\boldsymbol{E}\right]\cdot\boldsymbol{
\phi}d x\\
&= -\int_{\mathbb{R}^3}\nabla\times({\rm
i}k\boldsymbol{H})\cdot\boldsymbol{\phi}d
x+\int_{\mathbb{R}^3}\boldsymbol{E}\cdot\left(\nabla(\nabla\cdot\boldsymbol{\phi
})\right)d x+k^2\int_{\mathbb{R}^3}\boldsymbol{E}\cdot\boldsymbol{\phi}d x\\
&= -{\rm i}k\int_{\mathbb{R}^3}(\nabla\times\boldsymbol{H}+{\rm
i}k\boldsymbol{E})\cdot\boldsymbol{\phi}d x,
\end{align*} 
which implies that 
\[
\nabla\times\boldsymbol{H}=-{\rm i}k\boldsymbol{E}+\boldsymbol{J}. 
\]
Moreover, since
$\boldsymbol{E}\in\mathbb{X}\cap\boldsymbol{W}_{loc}^{-H+\epsilon,q}$, we have
for any $\boldsymbol{\phi}\in\boldsymbol{\mathcal{D}}$ that 
\begin{align*}
\left|\int_{\mathbb{R}^3}\boldsymbol{H}\cdot\boldsymbol{\phi}
dx\right| &= \left|\frac1{{\rm
i}k}\int_{\mathbb{R}^3}(\nabla\times\boldsymbol{E})\cdot\boldsymbol{\phi}
dx\right|=\frac1k\left|\int_{\mathbb{R}^3}\boldsymbol{E}
\cdot(\nabla\times\boldsymbol{\phi})dx\right|\\
&\le \frac1k\|\boldsymbol{E}\|_{\boldsymbol{W}^{-H+\epsilon,q}}
\|\nabla\times\boldsymbol{\phi}\|_{\boldsymbol{W}^{H-\epsilon,p}}\\
&\le \frac1k\|\boldsymbol{E}\|_{\boldsymbol{W}^{-H+\epsilon,q}}\|\boldsymbol{
\phi}\|_{\boldsymbol{W}^{H-\epsilon,p}(curl)},
\end{align*}
which completes the proof.
\end{proof}

It should be pointed out that if $\boldsymbol{J}\in\mathbb{X}$ satisfies
Assumption \ref{as:J} with $s\in(\frac3p-\frac12,\frac32]$, then it also holds
$\boldsymbol{J}\in\boldsymbol{W}^{H-\epsilon,p}$ with $H=s-\frac32$ and $p>1$
according to Lemma 2.6 in \cite{LW}, i.e., the assumptions in Theorem
\ref{tm:well-posed} are satisfied. If $\boldsymbol{J}\in\mathbb{X}$ satisfies
Assumption \ref{as:J} with $s\in(\frac32,\frac52)$, the current density 
$\boldsymbol{J}$ turns to be smoother
such that $\boldsymbol{J}\in\boldsymbol{C}^{0,\alpha}$ for all
$\alpha\in(0,s-\frac32)$ according to Lemma 2.6 in \cite{LW}. The well-posedness
of the problem in this case has been investigated in \cite{M03}. Therefore, we
only need to consider the current density $\boldsymbol{J}$ which satisfies
Assumption \ref{as:J}.

\section{Inverse scattering problem}

This section addresses the inverse scattering problem. According to Assumption
\ref{as:J}, the centered Gaussian random field $\boldsymbol{J}$ is determined by
its covariance operator $Q_{\boldsymbol{J}}^C$. To recover the strength matrix
$A(x)$ of the operator $Q_{\boldsymbol{J}}^C$, it is required to recover the
strength of the covariance operator between $J_j$ and $J_l$,
$j,l=1,2,3$, where $\boldsymbol{J}=(J_1,J_2,J_3)^\top$. For convenience,  we
denote by $a_{jl}( x)=a_{jl}^{\rm r}( x)+{\rm i}a_{jl}^{\rm i}(x)$ the
$(j,l)$-entry of the strength matrix $A(x)$. We discuss the covariance for each
component of $\boldsymbol J$ and the covariance between different components of
$\boldsymbol J$, separately.  

\subsection{Covariance for each component of $\boldsymbol{J}$}
\label{sec:3.1}

First, we consider the covariance operator for each component of
$\boldsymbol{J}$. By Theorem \ref{tm:well-posed}, the energy of each of the
components of $\boldsymbol{E}$ is 
\begin{align*}
\mathbb{E}|E_j(x)|^2 &= k^2\int_{\mathbb{R}^3}\int_{\mathbb{R}^3}\Phi_k( x,
y)\overline{\Phi_k( x, z)}\mathbb{E}[J_j( y)\overline{J_j( z)}]d y d
z\\
& = \frac{k^2}{(4\pi)^2}\int_{\mathbb{R}^3}\int_{\mathbb{R}^3}\frac{e^{{\rm i}k|
x- y|-{\rm i}k| x- z|}}{| x- y|| x- z|}C_{jj}( y, z)d y d z,
\end{align*}
where $C_{jj}, j=1, 2, 3$ is the $(j,j)$-entry of the kernel
$C_{\boldsymbol{J}}$.  

Let $C_{jj}=C_{jj}^{\rm r}+{\rm i}C_{jj}^{\rm i}$ where $C_{jj}^{\rm r}$ and
$C_{jj}^{\rm i}$ are the real and imaginary parts of $C_{jj}$, respectively.
It follows from Assumption \ref{as:J} that the principal symbols of $C_{jj}^{\rm
r}$ and $C_{jj}^{\rm i}$ are $a_{jj}^{\rm r}|\xi|^{-2s}$ and $a_{jj}^{\rm
i}|\xi|^{-2s}$, respectively.

\begin{theorem}\label{tm:Cjj}
Let Assumption \ref{as:J} hold and $\mathcal{U}\subset\mathbb{R}^3$
be a bounded open set which has a positive
distance to $\mathcal{O}$. For $j=1,2,3$, the strength $a_{jj}^{\rm r}$ is
uniquely determined by
\[
\lim_{k\to\infty}k^{2s-2}\mathbb{E}|E_j(x)|^2=\frac1{(4\pi)^2}\int_{\mathbb{R}^3}\frac1{|x-y|^2}a_{jj}^{\rm r}(y)dy,\quad x\in\mathcal{U}
\]
and $a_{jj}^{\rm i}\equiv0$.
\end{theorem}

\begin{remark}
The diagonal entry $a_{jj}$ of the strength matrix $A$ is a
real-valued function and it can be uniquely determined by the high frequency
limit of the phaseless data $\mathbb{E}|E_j|^2$ on an open set $\mathcal U$,
$j=1, 2, 3$. 
\end{remark}

\begin{proof}
Rewriting $\mathbb{E}|E_j(x)|^2$ through $C_{jj}^{\rm r}$ and $C_{jj}^{\rm i}$, one get
\begin{align*}
\mathbb{E}|E_j(x)|^2 & =
\frac{k^2}{(4\pi)^2}\int_{\mathbb{R}^3}\int_{\mathbb{R}^3 }\frac{\cos(k| x-
y|-k| x- z|)C_{jj}^{\rm r}( y, z)-\sin(k| x- y|-k| x- z|)C_{jj}^{\rm i}( y,
z)}{| x- y|| x- z|}d y d z\\
&\quad +\frac{{\rm
i}k^2}{(4\pi)^2}\int_{\mathbb{R}^3}\int_{\mathbb{R}^3}\frac{\sin(k| x- y|-k| x-
z|)C_{jj}^{\rm r}( y, z)+\cos(k| x- y|-k| x- z|)C_{jj}^{\rm i}( y, z)}{| x- y||
x- z|}d y d z,
\end{align*}
which apparently leads to
\begin{align}\label{eq:E1-1}
\mathbb{E}|E_j(x)|^2&=\frac{k^2}{(4\pi)^2}\Re\left[\int_{\mathbb{R}^3}\int_{
\mathbb{R}^3}\frac{e^{{\rm i}k(| x- y|-| x- z|)}C_{jj}^{\rm r}( y, z)}{| x- y||
x- z|}d y d z\right]\nonumber\\
&\quad
-\frac{k^2}{(4\pi)^2}\Im\left[\int_{\mathbb{R}^3}\int_{\mathbb{R}^3}\frac{ e^ {
{ \rm i}k(| x- y|-| x- z|)}C_{jj}^{\rm i}( y, z)}{| x- y|| x- z|}d y d z\right]
\end{align}
and
\begin{align}\label{eq:E1-2}
\Im\left[\int_{\mathbb{R}^3}\int_{\mathbb{R}^3}\frac{e^{{\rm i}k(| x- y|-| x- z|)}C_{jj}^{\rm r}( y, z)}{| x- y|| x- z|}d y d z\right]+\Re\left[\int_{\mathbb{R}^3}\int_{\mathbb{R}^3}\frac{e^{{\rm i}k(| x- y|-| x- z|)}C_{jj}^{\rm i}( y, z)}{| x- y|| x- z|}d y d z\right]=0.
\end{align}
It then suffices to consider the integrals 
\[
{\rm I}_1( x):=\int_{\mathbb{R}^3}\int_{\mathbb{R}^3}\frac{e^{{\rm i}k(| x- y|-|
x- z|)}C_{jj}^{\rm r}( y, z)}{| x- y|| x- z|}d y d z
\]
and 
\[
{\rm I}_2( x):=\int_{\mathbb{R}^3}\int_{\mathbb{R}^3}\frac{e^{{\rm i}k(| x- y|-|
x- z|)}C_{jj}^{\rm i}( y, z)}{| x- y|| x- z|}d y d z.
\]

The proof consists of four steps. 

Step 1. For any $ x\in\mathcal{U}$, by introducing a smooth function
$\theta\in C_0^{\infty}(\mathbb{R}^3)$ such that $\theta|_{\mathcal{U}}\equiv1$ and
supp$(\theta)\subset\mathbb{R}^3\backslash\overline{\mathcal{O}}$, we get
\begin{align*}
{\rm I}_1(x)=\int_{\mathbb{R}^3}\int_{\mathbb{R}^3}\frac{e^{{\rm i}k(| x- y|-|
x- z|)}}{| x- y|| x- z|}C_{jj}^{\rm r}( y, z)\theta( x)d y d z.
\end{align*}
Denote 
\[
S_1( y, z, x):=C_{jj}^{\rm r}( y, z)\theta(
x)=\frac1{(2\pi)^3}\int_{\mathbb{R}^3}e^{{\rm i}( y- z)\cdot\xi}s_1( y,
x,\xi)d\xi. 
\]
It can be easily verified that  the symbol $s_1( y, x,\xi)=c_{jj}^{\rm
r}(y,\xi)\theta( x)$, where $c_{jj}^{\rm r}( y,\xi)$ is the symbol of the
kernel $C_{jj}^{\rm r}$ (cf. \cite{LW}). By Assumption \ref{as:J}, the leading
term of $s_1$, which is the principal symbol of $S_1$, has the form
\[
s_1^p( y,\xi)=a_{jj}^{\rm r}( y)\theta( x)|\xi|^{-2s}. 
\]

Following \cite{LW}, we define an invertible
transformation $\tau:{\mathbb{R}}^9\to{\mathbb{R}}^9$ given by $
\tau( y, z, x)=(g,h,x)$,
where $g=(g_1,g_2,g_3)$ and $h=(h_1,h_2,h_3)$ with
\begin{eqnarray*}
&& g_1=\frac12\left(| x- y|-| x- z|\right),\quad
h_1=\frac12\left(| x- y|+| x- z|\right),\\
&&g_2=\frac12\bigg[| x- y|\arccos\Big(\frac{y_3-x_3}{| x- y|}
\Big)-| x- z|\arccos\Big(\frac{z_3-x_3}{| x- z|}\Big)\bigg],\\
&&h_2=\frac12\bigg[| x- y|\arccos\Big(\frac{y_3-x_3}{| x- y|}
\Big)+| x- z|\arccos\Big(\frac{z_3-x_3}{| x- z|}\Big)\bigg],\\
&&g_3=\frac12\bigg[| x- y|\arctan\Big(\frac{y_2-x_2}{y_1-x_1}
\Big)-| x- z|\arctan\Big(\frac{z_2-x_2}{z_1-x_1}\Big)\bigg],\\
&&h_3=\frac12\bigg[| x- y|\arctan\Big(\frac{y_2-x_2}{y_1-x_1}
\Big)+| x- z|\arctan\Big(\frac{z_2-x_2}{z_1-x_1}\Big)\bigg].
\end{eqnarray*}
Then 
\[
{\rm I}_1(x)=\int_{{\mathbb{R}}^3}\int_{{\mathbb{R}}^3}
e^{2{\rm i}k(e_1\cdot g)}S_2(g,h,x)dgdh,
\]
where $e_1=(1,0,0)$ and
\begin{align}\label{eq:S2}
S_2(g,h,x)=& S_1(\tau^{-1}(g,h,x))\frac{\text{det}\left((\tau^{-1})'(g,h,
x)\right)}{((g+h)\cdot e_1)((h-g)\cdot e_1)}\nonumber\\
=&:S_1(\tau^{-1}(g,h,x))L^\tau(g,h,x).
\end{align}

Step 2. To get an explicit expression of $S_2$ with respect to $(g,h,x)$,
we define another invertible transformation
$\eta:{\mathbb{R}}^9\to{\mathbb{R}}^9$ given 
by $\eta(y,z,x)=(v,w,x)$ with $v=y-z$ and $w=y+z$. Let the kernel 
\begin{align}\label{eq:S3}
S_3(v,w,x):=&~S_1\circ\eta^{-1}(v,w,x)=S_1\Big(\frac{v+w}2,\frac{w-v}2,x\Big)\nonumber\\
=&~\frac1{(2\pi)^3}\int_{\mathbb{R}^3}e^{{\rm
i}v\cdot\xi}s_1\Big(\frac{v+w}2,x,\xi\Big)d\xi\nonumber\\
=&~\frac1{(2\pi)^3}\int_{\mathbb{R}^3}e^{{\rm
i}v\cdot\xi}s_3\left(w,x,\xi\right)d\xi,
\end{align}
where we have used the properties of symbols in the last step (cf. \cite[Lemma
18.2.1]{H07}). More precisely, the symbol 
$s_3$ is defined by 
\[
s_3(w,x,\xi)=e^{\langle-{\rm
i}D_v,D_\xi\rangle}s_1\left(\frac{v+w}2,x,\xi\right)\bigg|_{v=0},
\]
which has an asymptotic expansion
\[
s_3(w,x,\xi)\sim\sum_{n=1}^{\infty}\frac{\langle-{\rm i}D_v,D_\xi\rangle^n}{n!}s_1\left(\frac{v+w}2,x,\xi\right)\bigg|_{v=0}.
\]
Hence the principal symbol of $s_3$ is 
\[
s_3^p(w,x,\xi)=s_1^p\left(\frac{v+w}2,x,\xi\right)\bigg|_{v=0}=a_{jj}^{r}\left(\frac
w2\right)|\xi|^{-2s}\theta(x). 
\]

Next, define the diffeomorphism
$\gamma:=\eta\circ\tau^{-1}:(g,h,x)\mapsto(v,w,x).$ It preserves the plane
$\{(g,h,x)\in\mathbb{R}^9 : g=0\}$, i.e., $\gamma(0,h,x)=(0,w,x)$. Now we are
able to consider the kernel $S_1\circ\tau^{-1}$ in \eqref{eq:S2}:
\[
S_1\circ\tau^{-1}(g,h,x)=S_1\circ\eta^{-1}\circ\eta\circ\tau^{-1}(g,h,x)=S_3\circ\gamma(g,h,x),
\]
where the kernel $S_3\circ\gamma$ admits a symbol $\tilde s_3(h,x,\xi)$ under the diffeomorphism $\gamma$ satisfying
\[
S_3\circ\gamma(g,h,x)=\frac1{(2\pi)^3}\int_{\mathbb{R}^3}e^{{\rm i}g\cdot\xi}\tilde s_3(h,x,\xi)d\xi.
\]
Comparing the above kernel $S_3\circ\gamma(g,h,x)$ with $S_3(v,w,x)$
defined in \eqref{eq:S3}, we may check that their symbols have the following
relationship (cf. \cite[Theorem 18.2.9]{H07} or \cite{LPS08}):
\begin{align*}
\tilde s_3(h,x,\xi) &= s_3\left(w(0,h,x),x,\left(\frac{\partial v}{\partial
g}(0,h,x)\right)^{-\top}\xi\right)\left|\det\left(\frac{\partial v}{\partial
g}(0,h,x)\right)\right|^{-1}+r_1(h,x,\xi)\\
& = s_3^p\left(w(0,h,x),x,\left(\frac{\partial v}{\partial
g}(0,h,x)\right)^{-\top}\xi\right)\left|\det\left(\frac{\partial v}{\partial
g}(0,h,x)\right)\right|^{-1}+r_2(h,x,\xi)\\
& = a_{jj}^{\rm r}\left(\frac{w(0,h,x)}2\right)\left|\left(\frac{\partial
v}{\partial
g}(0,h,x)\right)^{-\top}\xi\right|^{-2s}\left|\det\left(\frac{\partial
v}{\partial g}(0,h,x)\right)\right|^{-1}\theta(x)+r_2(h,x,\xi)\\
&=: \tilde s_3^p(h,x,\xi)+r_2(h,x,\xi),
\end{align*}
where the residuals $r_1,r_2\in\mathcal{S}^{-2s-1}$. Here $\mathcal{S}^{m}$
denotes the space of symbols of order $m$ (cf. \cite{H07}).

We conclude from the above discussions that
\begin{align*}
S_2(g,h,x) & = S_1(\tau^{-1}(g,h,x))L^{\tau}(g,h,x)\\
& = \frac1{(2\pi)^3}\int_{\mathbb{R}^3}e^{{\rm i}g\cdot\xi}\tilde
s_3(h,x,\xi)L^{\tau}(g,h,x)d\xi\\
& = \frac1{(2\pi)^3}\int_{\mathbb{R}^3}e^{{\rm i}g\cdot\xi}s_2(h,x,\xi)d\xi,
\end{align*}
where in the last step we have used the same property as that used in
\eqref{eq:S3}. Here the symbol $s_2$ satisfies
\[
s_2(h,x,\xi)=s_2^p(h,x,\xi)+r_3(h,x,\xi),
\]
where the residual $r_3\in\mathcal{S}^{-2s-1}$ and the principal symbol
\[
s_2^p(h,x,\xi)=\tilde s_3^p(h,x,\xi)L^{\tau}(0,h,x). 
\]

Step 3. Based on the expression of $S_2$, the integral $I_1(x)$ has the form
\begin{align*}
{\rm I}_1(x) & = \int_{{\mathbb{R}}^3}\int_{{\mathbb{R}}^3}
e^{2{\rm i}k(e_1\cdot g)}S_2(g,h,x)dgdh\\
& = \int_{{\mathbb{R}}^3}\int_{{\mathbb{R}}^3}
e^{2{\rm i}k(e_1\cdot g)}\left[\frac1{(2\pi)^3}\int_{\mathbb{R}^3}e^{{\rm i}g\cdot\xi}[\tilde s_3^p(h,x,\xi)L^{\tau}(0,h,x)+r_3(h,x,\xi)]d\xi\right]dgdh\\
& = \int_{{\mathbb{R}}^3}\int_{{\mathbb{R}}^3}
\left[\frac1{(2\pi)^3}\int_{\mathbb{R}^3}e^{{\rm i}g\cdot(\xi+2ke_1)}dg\right][\tilde s_3^p(h,x,\xi)L^{\tau}(0,h,x)+r_3(h,x,\xi)]d\xi dh\\
& = \int_{\mathbb{R}^3}[\tilde
s_3^p(h,x,-2ke_1)L^{\tau}(0,h,x)+r_3(h,x,-2ke_1)]dh.
\end{align*}
It then suffices to calculate 
\[
\tilde s_3^p(h,x,-2ke_1)=a_{jj}^{\rm r}\left(\frac{w(0,h,x)}2\right)\left|\left(\frac{\partial v}{\partial g}(0,h,x)\right)^{-\top}(-2ke_1)\right|^{-2s}\left|\det\left(\frac{\partial v}{\partial g}(0,h,x)\right)\right|^{-1}\theta(x)
\]
and
\[
L^{\tau}(0,h,x)=\frac{\left|\det\left((\tau^{-1})'(0,h,x)\right)\right|}{(h\cdot e_1)^2}.
\]

Noting that
\begin{eqnarray*}
&&h_1+g_1=|x-y|,\quad h_1-g_1=|x-z|,\\
&&\frac{h_2+g_2}{h_1+g_1}=\arccos\Big(\frac{y_3-x_3}{|x-y|}\Big),\quad\frac{
h_2-g_2}{h_1-g_1}=\arccos\Big(\frac{z_3-x_3}{|x-z|}\Big),\\
&&\frac{h_3+g_3}{h_1+g_1}=\arctan\Big(\frac{y_2-x_2}{y_1-x_1}\Big),\quad\frac
{h_3-g_3}{h_1-g_1}=\arctan\Big(\frac{z_2-x_2}{z_1-x_1}\Big),
\end{eqnarray*}
we get
\begin{align*}
y_1&= x_1+(h_1+g_1)\sin\left(\frac{h_2+g_2}{h_1+g_1}\right)\cos\left(\frac{
h_3+g_3}{h_1+g_1}\right),\\
y_2&= x_2+(h_1+g_1)\sin\left(\frac{h_2+g_2}{h_1+g_1}\right)\sin\left(\frac{
h_3+g_3}{h_1+g_1}\right),\\
y_3&= x_3+(h_1+g_1)\cos\left(\frac{h_2+g_2}{h_1+g_1}\right),\\
z_1&= x_1+(h_1-g_1)\sin\left(\frac{h_2-g_2}{h_1-g_1}\right)\cos\left(\frac{
h_3-g_3}{h_1-g_1}\right),\\
z_2&= x_2+(h_1-g_1)\sin\left(\frac{h_2-g_2}{h_1-g_1}\right)\sin\left(\frac{
h_3-g_3}{h_1-g_1}\right),\\
z_3&= x_3+(h_1-g_1)\cos\left(\frac{h_2-g_2}{h_1-g_1}\right).\\
\end{align*}
A simple calculation yields that 
\[
\frac{\partial v}{\partial g}(0,h,x)=
2\left[
\begin{array}{ccc}
\sin\alpha\cos\beta-\alpha\cos\alpha\cos\beta+\beta\sin\alpha\sin\beta&\cos
\alpha\cos\beta&-\sin\alpha\sin\beta\\
\sin\alpha\sin\beta-\alpha\cos\alpha\sin\beta-\beta\sin\alpha\cos\beta&\cos
\alpha\sin\beta&\sin\alpha\cos\beta\\
\cos\alpha+\alpha\sin\alpha&-\sin\alpha&0
\end{array}
\right],
\]
where $\alpha:=\frac{h_2}{h_1}, \beta:=\frac{h_3}{h_1}$, and
\[
(\tau^{-1})'(0,h,x)=\left[
\begin{array}{ccc}
\frac12\frac{\partial v}{\partial g}&\frac12\frac{\partial v}{\partial
g}&I\\[4pt]
-\frac12\frac{\partial v}{\partial g}&\frac12\frac{\partial v}{\partial g}&I\\
0&0&I
\end{array}
\right].
\]
Here $I$ is the $3\times 3$ identity matrix. It then leads to 
\[
\det\left(\frac{\partial v}{\partial g}(0,h,x)\right)=8\sin\alpha,
\]
\[
\left(\frac{\partial v}{\partial g}(0,h,x)\right)^{-\top}e_1=\frac12\left[\begin{array}{c}
\sin\alpha\cos\beta\\
\sin\alpha\sin\beta\\
\cos\alpha
\end{array}\right]
\]
and thus
\[
\tilde s_3^p(h,x,-2ke_1)=a_{jj}^{\rm r}\left(\frac{w(0,h,x)}2\right)\frac{k^{-2s}}{8|\sin\alpha|}\theta(x).
\]

To get $L^{\tau}(0,h,x)$, we next consider the matrix 
\[
(\tau^{-1})'(0,h,x)=\frac{\partial(y,z,x)}{\partial(g,h,x)}\bigg|_{g=0}=\left[
\begin{array}{ccc}
\frac12\frac{\partial v}{\partial g}&\frac12\frac{\partial v}{\partial
g}&I\\[4pt]
-\frac12\frac{\partial v}{\partial g}&\frac12\frac{\partial v}{\partial g}&I\\
0&0&I
\end{array}
\right],
\]
which gives $\det\left((\tau^{-1})'(0,h,x)\right)=8\sin^2\alpha$ and
\[
L^\tau(0,h,x)=\frac{8\sin^2\alpha}{(h\cdot e_1)^2}.
\]

Step 4. Based on the a priori estimates above, we obtain 
\begin{align*}
{\rm I}_1(x)=\int_{\mathbb{R}^3}\left[a_{jj}^{\rm
r}\left(\frac{w(0,h,x)}2\right)\frac{|\sin\alpha|}{(h\cdot
e_1)^2}k^{-2s}\theta(x)+r_3(h,x,-2ke_1)\right]dh,
\end{align*}
where $\frac{w(0,h,x)}2=(h_1\sin\alpha\cos\beta,h_1\sin\alpha\sin\beta,h_1\cos\alpha)+x$. 

Define another coordinate transform $\rho:\mathbb{R}^3\to\mathbb{R}^3$ by
\[
\rho(h)=\zeta:=(h_1\sin\alpha\cos\beta,h_1\sin\alpha\sin\beta,h_1\cos\alpha)+x.
\]
Noting that $|\zeta-x|=h_1=h\cdot e_1$ and
$\det((\rho^{-1})')=\frac1{\det(\rho')}$ with 
\[
\rho'=\left[
\begin{array}{ccc}
\sin\alpha\cos\beta-\alpha\cos\alpha\cos\beta+\beta\sin\alpha\sin\beta&\cos
\alpha\cos\beta&-\sin\alpha\sin\beta\\
\sin\alpha\sin\beta-\alpha\cos\alpha\sin\beta-\beta\sin\alpha\cos\beta&\cos
\alpha\sin\beta&\sin\alpha\cos\beta\\
\cos\alpha+\alpha\sin\alpha&-\sin\alpha&0
\end{array}
\right],
\] 
we have 
\[
{\rm I}_1(x)=\left[\int_{\mathbb{R}^3}\frac1{|\zeta-x|^2}a_{jj}^{\rm
r}(\zeta)d\zeta\right] k^{-2s}+O(k^{-2s-1}),\quad x\in\mathcal{U}.
\]

Following the same procedure as above, we may show that 
\begin{align*}
{\rm I}_2(x)=&\left[\int_{{\mathbb{R}}^3}\frac1{|\zeta-x|^2}a_{jj}^{\rm
i}(\zeta)d\zeta\right]k^{-2s}+O(k^{-2s-1}),\quad x\in\mathcal{U}.
\end{align*}
It follows from \eqref{eq:E1-1}--\eqref{eq:E1-2} that 
\[
\lim_{k\to\infty}k^{2s-2}\mathbb{E}|E_j(x)|^2=\frac1{(4\pi)^2}\int_{\mathbb{R}^3}\frac1{|\zeta-x|^2}a_{jj}^{\rm r}(\zeta)d\zeta
\]
and
\[
\int_{{\mathbb{R}}^3}\frac1{|\zeta-x|^2}a_{jj}^{\rm i}(\zeta)d\zeta=0,
\]
which imply that $a_{jj}^{\rm r}$ and $a_{jj}^{\rm i}$ can be uniquely
determined (cf. \cite[Theorem 4.6]{LW}) and in particular $a_{jj}^{\rm i}=0$.  
\end{proof}

\subsection{Covariance between different components of $\boldsymbol{J}$}

To recover the non-diagonal entries of the strength matrix $A(x)$, we now
consider the covariance between different components of $\boldsymbol{J}$. By
Theorem \ref{tm:well-posed}, we have 
\begin{align*}
\mathbb{E}[E_j(x)\overline{E_l(x)}]=&~k^2\int_{\mathbb{R}^3}\int_{
\mathbb{R}^3}\Phi_k( x, y)\overline{\Phi_k( x, z)}\mathbb{E}[J_j(
y)\overline{J_l( z)}]d y d z\\
=&~\frac{k^2}{(4\pi)^2}\int_{\mathbb{R}^3}\int_{\mathbb{R}^3}\frac{e^{{\rm i}k| x- y|-{\rm i}k| x- z|}}{| x- y|| x- z|}C_{jl}( y, z)d y d z
\end{align*}
for $j\neq l$ and $j,l=1,2,3$.
Denote by $C_{jl}^{\rm r}$ and $C_{jl}^{\rm i}$ the real and imaginary parts of
$C_{jl}$, respectively. The recovery of strengths $a_{jl}^{\rm r}$ and
$a_{jl}^{\rm i}$ of $C_{jl}^{\rm r}$ and $C_{jl}^{\rm i}$ are stated in the
following theorem. 

\begin{theorem}\label{tm:Cjl}
Let Assumption \ref{as:J} hold and $\mathcal{U}\subset\mathbb{R}^3$ be a bounded open set which has a positive
distance to $\mathcal{O}$. For $j,l=1,2,3$ and $j\neq l$, the strengths
$a_{jl}^{\rm r}$ and $a_{jl}^{\rm i}$ are uniquely determined by
\[
\lim_{k\to\infty}k^{2s-2}\Re\mathbb{E}[E_j(x)\overline{E_l(x)}]=\frac1{(4\pi)^2}\int_{\mathbb{R}^3}\frac1{|x-y|^2}a_{jl}^{\rm r}(y)dy,\quad x\in\mathcal{U}
\]
and 
\[
\lim_{k\to\infty}k^{2s-2}\Im\mathbb{E}[E_j(x)\overline{E_l(x)}]=\frac1{(4\pi)^2}\int_{\mathbb{R}^3}\frac1{|x-y|^2}a_{jl}^{\rm i}(y)dy,\quad x\in\mathcal{U}.
\]
\end{theorem}

\begin{remark}
The non-diagonal entry $a_{jl}$ of the strength matrix $A$ is a complex-valued
function and it can be uniquely determined by the high frequency limit
of the phased data $\mathbb{E}[E_j\overline{E_l}]$ on an open set $\mathcal U$ with
$j, l=1, 2, 3$ and $j\neq l$. 
\end{remark}

\begin{proof}
Using $C_{jl}^{\rm r}$ and $C_{jl}^{\rm i}$, we may split
$\mathbb{E}[E_j(x)\overline{E_l(x)}]$ into the real and imaginary parts
\begin{align*}
&\mathbb{E}[E_j(x)\overline{E_l(x)}]\\
& = \frac{k^2}{(4\pi)^2}\int_{\mathbb{R}^3}\int_{\mathbb{R}^3}\frac{\cos(k| x-
y|-k| x- z|)C_{jl}^{\rm r}( y, z)-\sin(k| x- y|-k| x- z|)C_{jl}^{\rm i}( y,
z)}{| x- y|| x- z|}d y d z\\
&\quad +\frac{{\rm
i}k^2}{(4\pi)^2}\int_{\mathbb{R}^3}\int_{\mathbb{R}^3}\frac{\sin(k| x- y|-k| x-
z|)C_{jl}^{\rm r}( y, z)+\cos(k| x- y|-k| x- z|)C_{jl}^{\rm i}( y, z)}{| x- y||
x- z|}d y d z\\
&=\frac{k^2}{(4\pi)^2}\left(\Re[{\rm I}_3(x)]-\Im[{\rm
I}_4(x)]\right)+\frac{{\rm
i}k^2}{(4\pi)^2}\left(\Im[{\rm I}_3(x)]+\Re[{\rm I}_4(x)]\right),
\end{align*}
where
\[
{\rm I}_3( x):=\int_{\mathbb{R}^3}\int_{\mathbb{R}^3}\frac{e^{{\rm i}k(| x- y|-|
x- z|)}C_{jl}^{\rm r}( y, z)}{| x- y|| x- z|}d y d z
\]
and 
\[
{\rm I}_4( x):=\int_{\mathbb{R}^3}\int_{\mathbb{R}^3}\frac{e^{{\rm i}k(| x- y|-|
x- z|)}C_{jl}^{\rm i}( y, z)}{| x- y|| x- z|}d y d z.
\]

Following the same procedure as that in the proof of Theorem \ref{tm:Cjj}, we
may show for any $x\in\mathcal{U}$ that 
\begin{align*}
{\rm I}_3(x)=\left[\int_{{\mathbb{R}}^3}\frac1{|\zeta-x|^2}a_{jl}^{\rm
r}(\zeta)\theta(x)d\zeta\right]k^{-2s}+O(k^{-2s-1})
\end{align*}
and
\begin{align*}
{\rm I}_4(x)=\left[\int_{{\mathbb{R}}^3}\frac1{|\zeta-x|^2}a_{jl}^{\rm
i}(\zeta)\theta(x)d\zeta\right]k^{-2s}+O(k^{-2s-1}).
\end{align*}
Consequently, we have for any $x\in\mathcal{U}$ that 
\[
\lim_{k\to\infty}k^{2s}\Im[{\rm
I}_3(x)]=\lim_{k\to\infty}k^{2s}\Im[{\rm I}_4(x)]=0
\]
and 
\begin{align*}
\lim_{k\to\infty}k^{2s}\Re[{\rm
I}_3(x)] &=\int_{{\mathbb{R}}^3}\frac1{|\zeta-x|^2}a_ { jl }^{\rm
r}(\zeta)\theta(x)d\zeta, \\
\lim_{k\to\infty}k^{2s}\Re[{\rm
I}_4(x)] &=\int_{{\mathbb{R}}^3}\frac1{|\zeta-x|^2} a_ { jl}^{\rm
i}(\zeta)\theta(x)d\zeta,
\end{align*}
which completes the proof.
\end{proof}

\begin{remark}
The above results can be combined into  
\begin{align}\label{eq:Ecompo}
\lim_{k\to\infty}k^{2s-2}\mathbb{E}[E_j(x)\overline{E_l(x)}]=\frac1{
(4\pi)^2}\int_{\mathbb{R}^3}\frac1{|x-y|^2}a_{jl}(y)dy,\quad
j,l=1,2,3,~x\in\mathcal{U}.
\end{align}
Equivalently, we have the matrix form 
\begin{align}\label{eq:Ematrix}
\lim_{k\to\infty}k^{2s-2}\mathbb{E}\left[\boldsymbol{E}(x)\boldsymbol{
E}^*(x)\right]=\frac1{(4\pi)^2}\int_{\mathbb{R}^3}\frac1{|x-y|^2}A(y)dy,
\quad x\in\mathcal{U}, 
\end{align}
which shows that the micro-correlation strength matrix function $A(x)$ can be 
uniquely determined by the high frequency limit of the data
$\mathbb{E}\left[\boldsymbol{E}\boldsymbol{E}^*\right]$ on
an open set $\mathcal{U}$. 
\end{remark}

\begin{remark}
If the covariance operators between components $J_j$ and $J_l$ are
pseudo-differential operators of the same order with the principal symbols
$a_{jl}(x)|\xi|^{-2s}$, then all the strength $\{a_{jl}\}_{j,l=1,2,3}$ can be
recovered at the same time by \eqref{eq:Ematrix}.

However, if the covariance operators between $J_j$ and $J_l$ are of different
orders with the principal symbols $a_{jl}(x)|\xi|^{-2s_{jl}}$ where $s_{jl}\in
[0, \frac{5}{2})$, then only the strength of the roughest term can be recovered
by \eqref{eq:Ematrix}. For example, if $s_{11}<s_{jl}$ for any $(j,l)\neq(1,1)$
and $j,l=1,2,3$, then the principal symbol of the covariance operator of
$\boldsymbol{J}$ is $A(x)|\xi|^{-2s_{11}}$ with
$A(x)=\text{diag}\{a_{11}(x),0,0\}$. In this case, the other strength
$a_{jl}(x)$ can be recovered by modifying \eqref{eq:Ecompo} as follows: 
\[
\lim_{k\to\infty}k^{2s_{jl}-2}\mathbb{E}[E_j(x)\overline{E_l(x)}]
=\frac1{(4\pi)^2}\int_{\mathbb{R}^3}\frac1{|x-y|^2}a_{jl}(y)dy,\quad
j,l=1,2,3,~x\in\mathcal{U}.
\]
\end{remark}

By Theorems \ref{tm:Cjj} and \ref{tm:Cjl}, we conclude that the strength matrix
$A(x)$ of the covariance operator $Q_{\boldsymbol{J}}^C$ can be uniquely
determined by the high frequency limit of the expectation of the electric field
$\boldsymbol{E}$ measured on an open set $\mathcal U$. Moreover, if only the
energy of the electric field $|E_j(x)|^2$, $j=1,2,3$, can be observed on an open
bounded domain $\mathcal{U}$, then the strength of $J_j$ can
be uniquely determined by a single realization of the phaseless data almost
surely, which is discussed in the following section.

\section{Recovery by a single path}

In this section, we present some ergodicity results to avoid using all the
sample paths in the recovery of the strength. We show that the diagonal entries
of the micro-correlation strength matrix can be uniquely determined
almost surely by the amplitude of the electric field averaged over the frequency
band at a single path. 

To indicate the dependence on the wavenumber $k$ of the electric field, we use
the notation $E_j(x;k)$ from now on. The following theorem is the main result
of this section. 

\begin{theorem}\label{tm:ergo}
Let Assumption \ref{as:J} hold and $\mathcal{U}\subset\mathbb{R}^3$
be a bounded open set which has a positive
distance to $\mathcal{O}$. The strength $a_{jj}$ is uniquely determined almost
surely by
\[
\lim_{K\to\infty}\frac1{K-1}\int_1^Kk^{2s-2}|E_j(x;k)|^2dk=\frac1{(4\pi)^2}\int_
{\mathbb{R}^3}\frac1{|x-y|^2}a_{jj}(y)dy,\quad x\in\mathcal{U}.
\]
\end{theorem}

The above theorem indicates that it is statistically stable to recover the
diagonal entries of the micro-correlation strength matrix since only a single 
realization is needed for the random source. We present some preliminaries on
ergodicity before showing the proof of Theorem \ref{tm:ergo}.

\subsection{Ergodic relation}

For $j=1, 2, 3$, define
\[
T_j(x):=\frac1{(4\pi)^2}\int_{\mathbb{R}^3}\frac1{|x-y|^2}a_{jj}(y)dy. 
\]
According to Theorem \ref{tm:Cjj}, it holds $a_{jj}=a_{jj}^{\rm r}+{\rm
i}a_{jj}^{\rm i}=a_{jj}^{\rm r}$ and 
\begin{align}\label{eq:limit}
\lim_{k\to\infty}k^{2s-2}\mathbb{E}|E_j(x;k)|^2=T_j(x),\quad x\in\mathcal{U},
\end{align}
which implies
\begin{align}\label{eq:limit-2}
\lim_{K\to\infty}\frac1{K-1}\int_1^Kk^{2s-2}\mathbb{E}|E_j(x;k)|^2dk=T_j(x).
\end{align}
In fact, for any $\epsilon>0$, it follows from \eqref{eq:limit} that there
exists some $k^*=k^*(\epsilon)>0$ such that 
\[
\left|k^{2s-2}\mathbb{E}|E_j(x;k)|^2-T_j(x)\right|<\frac{\epsilon}
2\quad\forall~k>k^*. 
\]
On the other hand, there exists $K^*=K^*(\epsilon)>0$ such that for any
$K>K^*$
\begin{align*}
&\left|\frac1{K-1}\int_1^K\left(k^{2s-2}\mathbb{E}|E_j(x;k)|^2-T_j(x)\right)dk\right|\\
& \le
\frac1{K-1}\int_1^{k^*}\left|k^{2s-2}\mathbb{E}|E_j(x;k)|^2-T_j(x)\right|dk
+\frac1{K-1}\int_{k^*}^K\left|k^{2s-2}\mathbb{E}|E_j(x;k)|^2-T_j(x)\right|dk\\
& \le \frac{C}{K-1}+\frac{K-k^*}{K-1}\frac{\epsilon}2<\epsilon
\end{align*}
for some constant $C>0$, and hence \eqref{eq:limit-2} holds.

To prove the result given in Theorem \ref{tm:ergo},
due to \eqref{eq:limit-2}, it then suffices to show 
\begin{align}\label{eq:error}
\lim_{K\to\infty}\frac1{K-1}\int_1^Kk^{2s-2}\left(|E_j(x;k)|^2-\mathbb{E}|E_j(x;k)|^2\right)dk=0.
\end{align}
The following propositions are required in order to get the ergodic relation
\eqref{eq:error}. The proofs can be found in \cite{CL04, LPS08, LHL}. 

\begin{proposition}\label{prop:CL}
Let $Y(t)$ be a centered random field with $\mathbb{E}[Y(t)] = 0$. If the
covariance function $R(\cdot,\cdot)$ is continuous and satisfies
\[
|R(t,u)|=\left|\mathbb{E}[Y(t)Y(u)]\right|\lesssim\frac{t^\alpha+u^\alpha}{
1+|t-u|^{\beta}},
\]
where the constants $\alpha, \beta$ satisfy $0\le2\alpha<\beta<1$, then
\[
\lim_{K\to\infty}\frac1{K-1}\int_1^KY(k)dk=0
\]
holds almost surely. 
\end{proposition}

\begin{proposition}\label{prop:LPS}
Let $X$ and $Y$ be centered Gaussian random variables with
$\mathbb{E}[X]=\mathbb{E}[Y]=0$. Then the following identity holds:
\[
\mathbb{E}\left[(X^2-\mathbb{E}[X^2])(Y^2-\mathbb{E}[Y^2])\right]=2(\mathbb{E}[XY])^2.
\]
\end{proposition}

\subsection{Proof of Theorem \ref{tm:ergo}}

Define
\[
Y_j(x;k):=k^{2s-2}\left(|E_j(x;k)|^2-\mathbb{E}|E_j(x;k)|^2\right),\quad x\in\mathcal{U}
\]
for $j=1,2,3$, which apparently satisfies $\mathbb{E}[Y(x;k)]=0$. Next is
to estimate $\mathbb{E}[Y_j(x;k_1)Y_j(x;k_2)]$ for any $k_1, k_2\geq 1$. 

Let $E_j=E_j^{\rm r}+{\rm i}E_j^{\rm i}$, $j=1,2,3$, where $E_j^{\rm r}$ and
$E_j^{\rm i}$ are the real and imaginary parts of $E_j$. A simple calculation
yields
\begin{align*}
Y_j(x;k)=k^{2s-2}\left((E_j^{\rm r}(x;k))^2-\mathbb{E}(E_j^{\rm r}(x;k))^2+(E_j^{\rm i}(x;k))^2-\mathbb{E}(E_j^{\rm i}(x;k))^2\right)
\end{align*}
and 
\begin{align}\label{eq:Y}
&\frac{|\mathbb{E}[Y_j(x;k_1)Y_j(x;k_2)]|}{k_1^{2s-2}k_2^{2s-2}}\nonumber\\
&=\mathbb{E}\Big[\left((E_j^{\rm r}(x;k_1))^2-\mathbb{E}(E_j^{\rm
r}(x;k_1))^2\right)\left((E_j^{\rm r}(x;k_2))^2-\mathbb{E}(E_j^{\rm
r}(x;k_2))^2\right)\Big]\nonumber\\
&\qquad +\mathbb{E}\Big[\left((E_j^{\rm r}(x;k_1))^2-\mathbb{E}(E_j^{\rm
r}(x;k_1))^2\right)\left((E_j^{\rm i}(x;k_2))^2-\mathbb{E}(E_j^{\rm
i}(x;k_2))^2\right)\Big]\nonumber\\
&\qquad +\mathbb{E}\Big[\left((E_j^{\rm i}(x;k_1))^2-\mathbb{E}(E_j^{\rm
i}(x;k_1))^2\right)\left((E_j^{\rm r}(x;k_2))^2-\mathbb{E}(E_j^{\rm
r}(x;k_2))^2\right)\Big]\nonumber\\
&\qquad +\mathbb{E}\Big[\left((E_j^{\rm i}(x;k_1))^2-\mathbb{E}(E_j^{\rm
i}(x;k_1))^2\right)\left((E_j^{\rm i}(x;k_2))^2-\mathbb{E}(E_j^{\rm
i}(x;k_2))^2\right)\Big]\nonumber\\
&=2\Big(\mathbb{E}[E_j^{\rm r}(x;k_1)E_j^{\rm
r}(x;k_2)]\Big)^2+2\Big(\mathbb{E}[E_j^{\rm r}(x;k_1)E_j^{\rm
i}(x;k_2)]\Big)^2\nonumber\\
&\qquad +2\Big(\mathbb{E}[E_j^{\rm i}(x;k_1)E_j^{\rm
r}(x;k_2)]\Big)^2+2\Big(\mathbb{E}[E_j^{\rm i}(x;k_1)E_j^{\rm
i}(x;k_2)]\Big)^2\nonumber\\
&=:{\rm I}_{j,1}+{\rm I}_{j,2}+{\rm I}_{j,3}+{\rm I}_{j,4},
\end{align}
where we have used Proposition \ref{prop:LPS}.

Using the facts
\[
\Re[g]\Re[h]=\frac12\Re[gh+g\overline{h}],
\]
\[
\Re[g]\Im[h]=-\Re[g]\Re[{\rm i}h]=-\frac12\Re[{\rm i}gh-{\rm
i}g\overline{h}]=\frac12\Im[gh-g\overline{h}],
\]
\[
\Im[g]\Im[h]=\Re[{\rm i}g]\Re[{\rm i}h]=\frac12\Re[g\overline{h}-gh]
\]
for any $g,h\in\mathbb{C}$, 
we get
\begin{align*}
{\rm I}_{j,1}&=\frac12\Re\left[\mathbb{E}[E_j(x;k_1)E_j(x;k_2)]+\mathbb{E}[
E_j(x;k_1)\overline{E_j(x;k_2)}]\right]^2\\
&\le\left|\mathbb{E}[E_j(x;k_1)E_j(x;k_2)]\right|^2+\left|\mathbb{E}[
E_j(x;k_1)\overline{E_j(x;k_2)}]\right|^2,
\end{align*}
\begin{align*}
{\rm I}_{j,2}&=\frac12\Im\left[\mathbb{E}[E_j(x;k_1)E_j(x;k_2)]-\mathbb{E}[
E_j(x;k_1)\overline{E_j(x;k_2)}]\right]^2\\
&\le\left|\mathbb{E}[E_j(x;k_1)E_j(x;k_2)]\right|^2+\left|\mathbb{E}[
E_j(x;k_1)\overline{E_j(x;k_2)}]\right|^2,
\end{align*}
\begin{align*}
{\rm I}_{j,3}&=\frac12\Im\left[\mathbb{E}[E_j(x;k_1)E_j(x;k_2)]-\mathbb{E}[
\overline { E_j(x;k_1)}E_j(x;k_2)]\right]^2\\
&\le\left|\mathbb{E}[E_j(x;k_1)E_j(x;k_2)]\right|^2+\left|\mathbb{E}[
E_j(x;k_1)\overline{E_j(x;k_2)}]\right|^2,
\end{align*}
\begin{align*}
{\rm
I}_{j,4}&=\frac12\Re\left[\mathbb{E}[E_j(x;k_1)\overline{E_j(x;k_2)}]-\mathbb {
E } [E_j(x;k_1)E_j(x;k_2)]\right]^2\\
&\le\left|\mathbb{E}[E_j(x;k_1)E_j(x;k_2)]\right|^2+\left|\mathbb{E}[
E_j(x;k_1)\overline{E_j(x;k_2)}]\right|^2.
\end{align*}
For $k_1,k_2\ge1$, let
\begin{align*}
\mathscr{A}_j(k_1,k_2)&=\left|\mathbb{E}[E_j(x;k_1)E_j(x;k_2)]\right|^2,\\
\mathscr{B}_j(k_1,k_2)&=\left|\mathbb{E}[E_j(x;k_1)\overline{E_j(x;k_2)}]
\right|^2.
\end{align*}
It suffices to estimate $\mathscr{A}_j(k_1,k_2)$ and $\mathscr{B}_j(k_1,k_2)$.

By Assumption \ref{as:J}, we may easily verify that 
\begin{align*}
&\mathbb{E}\left[E_j(x;k_1)E_j(x;k_2)\right]\\
=&~k_1k_2\int_{\mathbb{R}^3}\int_{\mathbb{R}^3}\Phi_{k_1}( x, y)\Phi_{k_2}( x, z)\mathbb{E}\left[J_j( y)J_j( z)\right]d y d z\\
=&~k_1k_2\int_{\mathbb{R}^3}\int_{\mathbb{R}^3}\Phi_{k_1}( x, y)\Phi_{k_2}( x, z)R_{jj}(y,z)d y d z\\
=&~0,
\end{align*}
where $R_{jj}$ is the $(j,j)$-entry of the relation kernel $R_{\boldsymbol{J}}$
of the relation operator $Q_{\boldsymbol{J}}^R$. Consequently, 
\[
\mathscr{A}_j(k_1,k_2)=0\quad \forall~k_1,k_2\ge1.
\]

For the term $\mathscr{B}_j(k_1,k_2)$, we have
\begin{align*}
\mathbb{E}\left[E_j(x;k_1)\overline{E_j(x;k_2)}\right]=&~k_1k_2\int_{\mathbb{R}^3}\int_{\mathbb{R}^3}\Phi_{k_1}( x, y)\overline{\Phi_{k_2}( x, z)}\mathbb{E}\left[J_j( y)\overline{J_j( z)}\right]d y d z\\
=&~\frac{k_1k_2}{(4\pi)^2}\int_{\mathbb{R}^3}\int_{\mathbb{R}^3}\frac{e^{{\rm i}(k_1| x- y|-k_2| x- z|)}}{| x- y|| x- z|}C_{jj}( y, z)d y d z.
\end{align*}
Noting
\[
k_1| x- y|-k_2| x- z|=(k_1+k_2)\frac{|x-y|-|x-z|}2+(k_1-k_2)\frac{|x-y|+|x-z|}2
\]
and using the coordinate transform and the symbols defined in the proof of
Theorem \ref{tm:Cjj}, we get
\begin{align*}
\mathscr{B}_j(k_1,k_2)&=\left|\frac{k_1k_2}{(4\pi)^2}\int_{\mathbb{R}^3}\int_{
\mathbb{R}^3}e^{{\rm i}\left((k_1+k_2)e_1\cdot g+(k_1-k_2)e_1\cdot
h\right)}S_2(g,h,x)dgdh\right|^2\\
&=\left|\frac{k_1k_2}{(4\pi)^2}\int_{\mathbb{R}^3}\int_{\mathbb{R}^3}e^{{\rm
i}\left((k_1+k_2)e_1\cdot g+(k_1-k_2)e_1\cdot
h\right)}\left[\frac1{(2\pi)^3}\int_{\mathbb{R}^3}e^{{\rm
i}g\cdot\xi}s_2(h,x,\xi)d\xi\right]dgdh\right|^2\\
&=\left|\frac{k_1k_2}{(4\pi)^2}\int_{{\mathbb{R}}^3}\int_{{\mathbb{R}}^3}e^{{\rm
i}(k_1-k_2)e_1\cdot h}
\left[\frac1{(2\pi)^3}\int_{\mathbb{R}^3}e^{{\rm i}g\cdot(\xi+(k_1+k_2)e_1)}dg\right]s_2(h,x,\xi)d\xi dh\right|^2\\
&=\left|\frac{k_1k_2}{(4\pi)^2}\int_{{\mathbb{R}}^3}e^{{\rm i}(k_1-k_2)e_1\cdot
h}s_2(h,x,-(k_1+k_2)e_1)dh\right|^2.
\end{align*}
If $|k_1-k_2|<1$, due to the fact that $A(x)$ is compactly supported, then we
obtain
\begin{align*}
\mathscr{B}_j(k_1,k_2)&=\bigg|\frac{k_1k_2}{(4\pi)^2}\int_{\mathbb{R}^3}e^{{\rm
i}(k_1-k_2)e_1\cdot
h}\bigg[a_{jj}\left(\frac{w(0,h,x)}2\right)\frac{|\sin(h_2/h_1)|}{(h\cdot
e_1)^2}\left(\frac{k_1+k_2}2\right)^{-2s}\theta(x)\\
&\quad+r_3(h,x,-(k_1+k_2)e_1)\bigg]dh\bigg|^2\\
&\lesssim\left(\frac{k_1k_2}{(k_1+k_2)^{2s}}\right)^2,
\end{align*}
If $|k_1-k_2|\ge1$, then for arbitrary $\beta\in(0,1)$, we deduce that 
\begin{align*}
\mathscr{B}_j(k_1,k_2)=&\left|\frac{k_1k_2}{(4\pi)^2}\frac1{{\rm i}(k_1-k_2)}\int_{{\mathbb{R}}^3}s_2(h,x,-(k_1+k_2)e_1)de^{{\rm i}(k_1-k_2)h_1}dh_2dh_3\right|^2\\
=&\left|\frac{k_1k_2}{(4\pi)^2}\frac1{{\rm i}(k_1-k_2)}\int_{{\mathbb{R}}^3}e^{{\rm i}(k_1-k_2)h_1}\partial_{h_1}s_2(h,x,-(k_1+k_2)e_1)dh_1dh_2dh_3\right|^2\\
\lesssim&\left(\frac{k_1k_2}{(k_1+k_2)^{2s}|k_1-k_2|}\right)^2\\
\le&\left(\frac{k_1k_2}{(k_1+k_2)^{2s}}\right)^2\frac1{|k_1-k_2|^\beta},
\end{align*} 
since the symbol $s_2$ is also compactly supported and
$|\partial_{h_1}s_2(h,x,\xi)|\lesssim|\xi|^{-2s}$ for $x\in\mathcal{U}$. We 
conclude from the above estimates that
\[
\mathscr{A}_j(k_1,k_2)+\mathscr{B}_j(k_1,k_2)\lesssim\frac{k_1^2k_2^2}{(k_1+k_2)^{4s}}\frac1{1+|k_1-k_2|^\beta}.
\]

Finally, it follows from \eqref{eq:Y} that 
\begin{align*}
|\mathbb{E}[Y_j(x;k_1)Y_j(x;k_2)]|
&\le 4k_1^{2s-2}k_2^{2s-2}\left(\mathscr{A}_j(k_1,k_2)+\mathscr{B}_j(k_1,
k_2)\right)\\
&\lesssim
\left(\frac{k_1k_2}{(k_1+k_2)^2}\right)^{2s}\frac1{1+|k_1-k_2|^\beta}\\
&\lesssim \frac1{1+|k_1-k_2|^\beta}.
\end{align*}
Using Proposition \ref{prop:CL} with $\alpha=0$, we get
\[
\lim_{K\to\infty}\frac1{K-1}\int_1^{K}Y_j(x;k)dk=0\quad\forall~x\in\mathcal{U},
\] 
and hence \eqref{eq:error} holds. This completes the proof of Theorem \ref{tm:ergo}.

\section{Conclusion}

In this paper, we have studied the three-dimensional Maxwell's equations driven by
a rough complex-valued Gaussian vector field, where the covariance operator of
the random source is a pseudo-differential operator with a complex-valued
strength matrix. Under an appropriate assumption of the random source, the
well-posedness of the direct scattering problem is established in the
distribution sense. The regularity of the electromagnetic field is also given.
The micro-correlation strength matrix of the random source is shown to be
uniquely determined by the high frequency limit of the expectation of the
electric field. Moreover, the diagonal entries of the strength matrix are shown
to be uniquely determined by the amplitude of the electric field averaged over
the frequency band at a single path due to the ergodicity. 

In this work, we assume that the real and imaginary parts of the random source
are independent and identically distributed, i.e., they are uncorrelated. A 
possible future work is to remove the assumption and consider more general
complex-valued Gaussian vector fields where the real and imaginary parts are
correlated. In this case, the centered random source would be determined by not
only its covariance operator but also its relation operator. The recovery of the
strength matrix of the relation operator is open since the micro-local analysis
seems not work anymore. The same issue appears in the inverse elastic wave
scattering problem. If the random source is a real-valued Gaussian vector field
and the components are independent and identically distributed, the recovery of
the scalar strength function for the elastic scattering problem has been
investigated in \cite{LHL,LL}. However, there are no results on the problem if
the random source is complex and correlated. It is also unclear how the
non-diagonal entries of the strength matrix can be uniquely determined by only
the amplitude of the electric field averaged over the frequency band at a
single path. We hope to be able to report the progress on these problems
elsewhere in the future.

\end{document}